\documentclass[oneside,11pt]{amsart}
\usepackage{mathpazo}
\usepackage{geometry}
\usepackage[parfill]{parskip}
\usepackage{amssymb}
\newtheorem{theorem}{Theorem}
\newtheorem{definition}{Definition}

\newtheorem*{PropositionI}{Galois' Proposition I}
\usepackage{tikz}

\title[Symmetric Polynomials]{The Fundamental Theorem on Symmetric Polynomials: History's First Whiff of Galois Theory}
\author{Ben Blum-Smith}
\address{Ben Blum-Smith\\Courant Institute of Mathematical Sciences\\New York University\\251 Mercer St.\\New York, NY 10012}
\email{ben@cims.nyu.edu}
\author{Samuel Coskey}
\address{Samuel Coskey\\Department of Mathematics\\Boise State University\\1910 University Dr\\Boise, ID 83725}
\email{scoskey@nylogic.org}
\urladdr{boolesrings.org/scoskey}

\begin{document}

\maketitle

Evariste Galois' (1811--1832) short life is one of the classic romantic tragedies of mathematical history.  The teenage Galois developed a revolutionary theory of equations, answering more fully than ever before a centuries-old question: why can't we find a formula for solving quintic polynomials analogous to the quadratic, cubic and quartic formulas?  Then he died in a duel, probably over the honor of a woman \cite[p.\ 290]{stillwell}, before his twenty-first birthday.  His discoveries lay in obscurity for 14 years, until Joseph Liouville encountered them, recognized their importance, and made them known \cite[p.\ 1]{edwards}, \cite[p.\ 290]{stillwell}.  Over the next few decades, the ideas Galois introduced -- groups and fields -- transcended the problem they were designed to solve, and reshaped the landscape of modern mathematics.

This story is told and retold in popularizations of mathematics.  Less frequently discussed is the actual content of Galois' discoveries.  This is usually reserved for a course in advanced undergraduate or graduate algebra.  This article is intended to give the reader a little of the flavor of Galois' work through a theorem that plays a unique role in it.  This theorem appears to have been understood, or at least intuited and used, by Newton, as early as 1665.  By the turn of the nineteenth century it was regarded as well known.  For Galois himself, it was the essential lemma on which his entire theory rested.  This theorem is now known as the Fundamental Theorem on Symmetric Polynomials (FTSP).

This essay has three goals: the first expository, the second pedagogical, and the third mathematical. Our expository goal is to articulate the central insight of Galois theory---the connection between \emph{symmetry} and \emph{expressibility}, described below---by examining the FTSP and its proof.  Here we intend to reach any mathematics students or interested laypeople who have heard about this mysterious ``Galois theory'' and wish to know what it's all about. Our point of view (elaborated in Sections~\ref{backstory} and~\ref{ftspingalois}) is that the FTSP manifests the central insights of the theory, so that the interested reader can get a little taste of Galois theory from this one theorem alone.

We also wish to reach readers who have studied Galois theory but feel they missed the forest for the trees. After all Galois theory has been substantially reformulated since Galois' time, and only the modern formulation is typically treated in university classes.  For example, Galois' reliance on the FTSP has been replaced with the elementary theory of vector spaces over a field, a theory unavailable in the 1820's.  A student of the modern theory may not even immediately recognize what we are calling the central insight---the connection between symmetry and expressibility---in what they have learned. In Section~\ref{ftspingalois} we address this by placing the FTSP in the context of the theorems Galois used it to prove, and link these in turn with the modern formulation.

Our pedagogical aim comes from the approach we take to the theorem. Our narrative arose out of an informal inquiry-based course in group theory and the historical foundations of Galois theory.\footnote{This course was given by Ben in 2009--10 to a small group of teachers and mathematicians including Samuel, Kayty Himmelstein, Jesse Johnson, Justin Lanier, and Anna Weltman.}  In it, we posed the problem of trying to give a na\"ive proof of the theorem before learning the classical proof. In Sections~\ref{backstory} and~\ref{twoandthree}, we describe the participants' encounter with this problem, and in doing so we hope to showcase the pleasure of mathematical discovery, as well as provide a classroom module for other instructors and students.

Our mathematical goals arise directly from this pedagogical experience. The classical proof of the FTSP, given in Section~\ref{classical}, involves a clever trick that diverges from the participants' proof ideas and is therefore, from a pedagogical standpoint, a bit of a \emph{deus ex machina}. The participants' work in the course inspired us to develop a new proof that replaces this trick with another method (Section~\ref{spreadness}), or really any of a family of other methods (Section~\ref{spreadnessandlex}), that are more consonant with the direction of the participants' thinking. Our view is that the new arguments shed light on what the classical proof was really doing all along. The explication of these proof variants and their relationship to the classical proof is our mathematical aim.


We would like to thank Benjamin Weiss for his assistance in tracking the history of the FTSP, and especially the provenance of the classical proof given in Section~\ref{classical}; Harold Edwards for a clarifying conversation about Galois' Proposition~I, discussed in Section~\ref{ftspingalois}; and Walter Stromquist and several anonymous referees for very helpful comments.

\section{The back story}\label{backstory}

The FTSP states that any polynomial in $n$ variables $x_1, \dots, x_n$ that is invariant under all permutations of the variables (i.e.,\ \emph{symmetric}) is representable in a unique way as a polynomial in the $n$ \emph{elementary symmetric polynomials},
\begin{align*}
\sigma_1&=x_1+\dots+x_n\\
\sigma_2&=x_1x_2+x_1x_3+\dots+x_{n-1}x_n=\sum_{i<j} x_ix_j\\
\sigma_3&=\sum_{i<j<k} x_ix_jx_k\\
&\;\;\vdots\\
\sigma_n&=x_1x_2\dots x_n
\end{align*}
Formally, it says:

\begin{theorem}[Fundamental Theorem on Symmetric Polynomials]
  Any symmetric polynomial in $n$ variables $x_1,\ldots,x_n$ is representable in a unique way as a polynomial in the elementary symmetric polynomials $\sigma_1,\ldots,\sigma_n$.
\end{theorem}

For example, since the polynomial $d=(x_1-x_2)^2$ is unchanged by transposing the two variables, the theorem guarantees $d$ can be expressed in terms of $\sigma_1=x_1+x_2$ and $\sigma_2=x_1x_2$.  In this case the expression is easy to find: $d=(x_1+x_2)^2-4x_1x_2=\sigma_1^2-4\sigma_2$.

The importance of the theorem to the theory of equations stems from the fact known as \emph{Vieta's theorem}, that the coefficients of a single-variable polynomial are precisely the elementary symmetric polynomials in its roots:

\begin{theorem}[Vieta's Theorem]
  Let $p(z)$ be an $n^\mathrm{th}$ degree monic polynomial with roots $\alpha_1,\alpha_2, \dots, \alpha_n$.  Let $\sigma_1,\dots,\sigma_n$ be the $n$ elementary symmetric polynomials in the $\alpha_i$.  Then
  \[p(z)=z^n-\sigma_1z^{n-1}+\sigma_2z^{n-2}-\dots+(-1)^n\sigma_n
  \]
\end{theorem}

The proof is a straightforward computation, but its ease belies its significance.  With this fact in hand, the FTSP becomes the statement that given any polynomial equation $p(z)=0$, any symmetric polynomial in its roots is actually a polynomial in its coefficients, which can be written down without (in fact, on the way to) solving the equation.  Continuing the example from above, if $x_1$ and $x_2$ are the roots of a monic quadratic polynomial, then that polynomial is $p(z)=z^2-\sigma_1z+\sigma_2$. The theorem guarantees that the discriminant $d$ of $p(z)$ (defined as the square of the difference between the roots) would have an expression in terms of the coefficients.  This of course is key to the quadratic's solution: $\sqrt{d}$ is the difference between the roots and $\sigma_1$ is the sum of the roots; and the roots themselves can be deduced from these two values.  Since $d$ can be expressed in terms of the coefficients, it follows that the roots can too.

This is the form in which the FTSP played its seminal historical role.  As mentioned above, it appears to have been at least intuited by Newton \cite[pp.~6--8]{edwards}, who also gave a formula (now known as Newton's Theorem) that effectively proves the special case of power sums.\footnote{Newton's Theorem states that if $p_j = \sum x_i^j$ is the $j$th power sum, then the power sums and elementary symmetric polynomials together obey the relation $p_k - p_{k-1}\sigma_1 + p_{k-2}\sigma_2 - \dots \pm k\sigma_k = 0$. (If $k>n$, we are to interpret $\sigma_j$ as zero for $j>n$.) This formula allows one recursively to construct formulas for the power sums in terms of the elementary symmetric polynomials.} The result embedded itself in the common knowledge of mathematicians over the course of the eighteenth century, through the calculations of many different people (\cite[p.\ 99]{tignol}, \cite[pp.~7-8]{edwards}). For a discussion of some of its historical applications prior to Galois' work, see Greg St.~George's delightful essay ``Symmetric Polynomials in the Work of Newton and Lagrange" in \emph{Mathematics Magazine} \cite{greg}.


As mentioned above, the FTSP brings out one of the central insights of Galois theory, which is the connection between \emph{symmetry} and \emph{rational expressibility}.  We have a polynomial $p(z)$, whose coefficients we know.  Even if we don't know the roots, the FTSP tells us that \emph{symmetric} expressions in the roots are \emph{rationally expressible} in terms of the coefficients.  As a corollary, if the coefficients of $p(z)$ are rational numbers, then every symmetric expression in the roots (for instance the sum of their squares) has a rational value as well. Symmetry guarantees rational expressibility.  In the last section we will indicate how this fits into the bigger picture of Galois theory.

In our course on Galois theory, we did not approach the FTSP directly, but rather sidled up to it by considering some problems of historical significance that implicitly depend on it.  The first was a problem of Newton: given two polynomials $f,g$, how can one determine whether they have a root in common without finding the roots?  (This problem is discussed at length in Greg St.~George's essay.)  The second was posed by Gauss in his \emph{Disquisitiones Arithmeticae}: given a polynomial $f$, without finding its roots, determine a polynomial $g$ whose roots are the squares, or cubes, etc.,\ of the roots of $f$. 

Participants solved both of these problems for polynomials of low degree.  The solutions were accomplished by writing desired expressions in the roots, which turned out to be symmetric, and then expressing these in terms of the coefficients instead.  For example, they considered Gauss' problem for a quadratic: if $f(z)=z^2-\sigma_1z+\sigma_2$, write down $g$ whose roots are the squares of $f$'s.  In this case, if $\alpha_1,\alpha_2$ are the roots of $f$, then $\alpha_1^2,\alpha_2^2$ are the roots of $g$, so that
\[
g=(z-\alpha_1^2)(z-\alpha_2^2)=z^2-(\alpha_1^2+\alpha_2^2)z+\alpha_1^2\alpha_2^2
\]
To write down this polynomial without actually solving $f$, it would be necessary to have expressions for the coefficients $\alpha_1^2+\alpha_2^2$ and $\alpha_1^2\alpha_2^2$ in terms of $f$'s coefficients $\sigma_1$ and $\sigma_2$.  You may enjoy looking for them yourself before reading the next line.
\begin{align*}
\alpha_1^2+\alpha_2^2 &= \sigma_1^2-2\sigma_2\\
\alpha_1^2\alpha_2^2 &= \sigma_2^2
\end{align*}

Participants were able to find such expressions in every case we considered, and so began to suspect that something like the FTSP would be true.  It was clear that any expression in the roots of a polynomial would have to be symmetric to be expressible in terms of the coefficients, since the coefficients are already symmetric.  But it was not clear that any symmetric expression in the roots would be expressible in the coefficients.

\section{The two and three variable case}\label{twoandthree}

In this section we begin to approach the question of why any symmetric expression in the roots is expressible in terms of the coefficients from the na\"ive point of view of the participants.  It is natural to begin with the special cases in which the polynomial has just two and then three variables.  The participants were able to cobble together proofs in these two cases over the course of two meetings.

To start, let $p(x,y)$ be a polynomial which is symmetric in $x$ and $y$.  We want to show that it can be expressed as a polynomial in $\sigma_1=x+y$ and $\sigma_2=xy$.  Taking an arbitrary monomial $x^my^n$ which appears in $p(x,y)$, we will try to ``take care of it'' by expressing it in terms of $\sigma_1$ and $\sigma_2$.  Renaming the variables if necessary, we can suppose that $m\geq n$.  If $n>0$, then we can already write $x^my^n$ as $\sigma_2^nx^{m-n}$, so it suffices to deal with monomials of the form $x^n$.  For this, note that the symmetry of $p(x,y)$ implies its conjugate monomial $y^n$ is also a term of $p(x,y)$, so we can deal with $x^n+y^n$ together.  Now, we recognize $x^n+y^n$ as the first and last terms of $\sigma_1^n=(x+y)^n$.  Hence, we have that
\begin{align*}
  x^n+y^n&=\sigma_1^n-\binom{n}{1}xy^{n-1}-\cdots-\binom{n}{n-1}x^{n-1}y\\
         &=\sigma_1^n-\sigma_2q(x,y)\;,
\end{align*}
where $q(x,y)$ is a polynomial of degree $n-2$.  This shows that an induction on the degree of $p(x,y)$ will succeed.

In the case of three variables, let $p(x,y,z)$ be a polynomial which is symmetric in $x,y,z$.  We wish to express $p(x,y,z)$ as a function of $\sigma_1=x+y+z$, $\sigma_2=xy+xz+yz$, and $\sigma_3=xyz$.  Again consider an arbitrary monomial $x^my^nz^p$ in $p(x,y,z)$, where for convenience we assume that $m\geq n\geq p$.  If $p>0$ then we can write $x^my^nz^p$ as $\sigma_3^px^{m-p}y^{n-p}$, leaving a monomial with just two variables to deal with.  In other words, we only need to treat monomials of the form $x^my^n$.  Now, all of the conjugate monomials $x^nz^m$, $x^mz^n$, $x^nz^m$, $y^mz^n$ and $y^nz^m$ are also found in $p(x,y,z)$.  In analogy to the two variable case, we now recognize that these are all terms of
\[\sigma_1^{m-n}\sigma_2^n=(x+y+z)^{m-n}(xy+xz+yz)^n\;.
\]
Thus, we can write
\[x^my^n+x^nz^m+x^mz^n+x^nz^m+y^mz^n+y^nz^m=\sigma_1^{m-n}\sigma_2^n-q(x,y,z)\;.
\]
Unlike the two variable case, the leftover terms (which we denoted $q(x,y,z)$) need not have a common factor.  However, any term of $q(x,y,z)$ which happens to involve just two variables must be a conjugate of $x^ky^l$, where $m>k\geq l>n$ and $k+l=m+n$.  So while we have not reduced the degree in every case, in the cases where we have not we have nonetheless improved the situation in one key way: we have reduced the \emph{spread} between the exponents.  In other words, this time we will succeed using an induction which takes into account both the degree \emph{and} the spread between the exponents in the case of monomials with just two variables.

It is natural to try to generalize this method to four and more variables, but there are some difficulties.  For starters, it is not clear what the ``spread between the exponents" would \emph{mean} when there are more than two variables in play!  While it would have been nice to let the discussion unfold and try to turn this into a general proof, the instructor (Ben) decided in the interest of time to wrap up the FTSP by presenting one of the standard arguments.

\section{A classical proof}\label{classical}

In this section we present a classical proof of the FTSP. Our presentation follows that of Sturmfels \cite{sturmfels}.  The proof itself goes back at least to Gauss.\footnote{Tignol \cite[p.~99]{tignol} credits Waring's 1770 \emph{Meditationes Algebraicae} \cite{waring} for this proof, since it includes the key construction of Equation~\eqref{specialprod}. However Waring does not mention the lexicographic order or argue that the algorithm terminates. Gauss' proof \cite[pp.~36--37]{gauss}, which does both of these, is dated 1815. We would like to thank Benjamin Weiss for referring us to Gauss' proof.} In the next section, we will return to the participants' proof idea. 



\begin{proof}[Proof of the FTSP]
  Let $f$ be the symmetric polynomial to be represented.  The set of $f$'s terms of a given degree is itself a symmetric polynomial and if we can represent each of these as a polynomial in the $\sigma_i$, we can represent $f$; thus nothing is lost by assuming that $f$ is homogeneous.

  Now, order the terms of $f$ lexicographically.  That is, put the term with the highest power of $x_1$ \emph{first}, and if there is a tie, decide in favor of the term with the most $x_2$, and so on.  Formally, define $ax_1^{i_1}x_2^{i_2}\dots x_n^{i_n}>bx_1^{j_1}x_2^{j_2}\dots x_n^{j_n}$ if $i_1>j_1$, or if $i_1=j_1$ and $i_2>j_2$, or if $i_1=j_1$, $i_2=j_2$ and $i_3>j_3$, etc.,\ and then order the terms of $f$ so that the first term is $>$ the second which is $>$ the third, and so on.

  Because $f$ is symmetric, for every term $cx_1^{i_1}x_2^{i_2}\dots x_n^{i_n}$ in it, it also contains all possible terms that look like this one except with the exponents permuted (its ``conjugates").  It follows that the leading term of $f$, say $c_1x_1^{i_1}x_2^{i_2}\dots x_n^{i_n}$, has $i_1\geq i_2\geq \dots \geq i_n$.  We let 

\begin{equation}\label{specialprod}
g_1=c_1\sigma_1^{i_1-i_2}\sigma_2^{i_2-i_3}\cdots\sigma_{n-1}^{i_{n-1}-i_n}\sigma_n^{i_n}\;.
\end{equation}
Then $g_1$ is symmetric, and it is easy to see that it has the same leading term as $f$.  Thus $f-g_1$ is symmetric with a ``lower'' leading term, which we denote $c_2x_1^{j_1}x_2^{j_2}\cdots x_n^{j_n}$.  As before, it follows from the symmetry that $j_1\geq j_2\geq\cdots\geq i_n$.  Thus we can let $g_2=c_2\sigma_1^{j_1-j_2}\sigma_2^{j_2-j_3}\cdots\sigma_n^{j_n}$, so that $g_2$ has the same leading term as $f-g_1$, and $f-g_1-g_2$ has a leading term that is lower still.

  Continue in like manner.  The algorithm must eventually terminate with no terms remaining, because there are only finitely many possible monomials $x_1^{i_1}x_2^{i_2}\cdots x_n^{i_n}$ of a given degree in the first place.  Thus we must come to a point where we have $f-g_1-g_2-\cdots-g_k=0$.  Then $f=g_1+g_2+\cdots+g_k$ is the desired representation of $f$ as a polynomial in the $\sigma_i$.  

  To prove its uniqueness, it is sufficient to show that the zero polynomial in $x_1,\ldots,x_n$ is representable uniquely as the zero polynomial in $\sigma_1,\ldots,\sigma_n$.  This is so because no two distinct products of elementary polynomials $\sigma^{k_1}\cdots\sigma^{k_n}$ have the same leading term.  (The leading term of $\sigma_1^{k_1}\cdots\sigma_n^{k_n}$ is $x_1^{k_1+\dots+k_n}x_2^{k_2+\dots+k_n}\cdots x_n^{k_n}$, and the map $(k_1,\ldots,k_n)\mapsto(k_1+\cdots+k_n,\ldots,k_{n-1}+k_n,k_n)$ is injective.)  Thus the leading terms in a sum of distinct products of elementary symmetric polynomials cannot cancel; so such a sum cannot equal zero unless it is empty.
\end{proof}

This lexicographic-order argument is elegant, simple, and highly constructive. From a pedagogical standpoint, however, it depends on a very counterintuitive move. Lexicographic order (lex, for short) is a total order on the set of monomials.\footnote{In fact, it is even better than this: it is a \emph{monomial order}, i.e.,\ a well-order that is compatible with multiplication.} It determines a unique leading term in any polynomial, and in fact this is (prima facie) part of how the proof works.  The proof conjures in one's mind an image of the terms of $f$ totally ordered and then picked off one-by-one, left to right, by our careful choice of $g_1,\ldots,g_k$.

However, since $f$ and $g_1,\ldots,g_k$ are all symmetric, the terms are not really being picked off one at a time.  Forming $f-g_1$ not only cancels the leading term $c_1x_1^{i_1}x_2^{i_2}\cdots x_n^{i_n}$, but all of its conjugates as well (for instance, the ``trailing term" $c_1x_1^{i_n}x_2^{i_{n-1}}\cdots x_n^{i_1}$).  Somehow, the lex ordering obscures the symmetry between the conjugates by distinguishing one of them as the leading term, even while it exploits this symmetry to make the proof work.

In this way it diverges sharply from the participants' naive attempts to prove the theorem, all of which dealt with all the monomials in a given conjugacy class on an equal footing. This makes the appeal to lex order highly unexpected, which is part of the proof's charm, but it also raises the (essentially mathematical, but pedagogically resonant) question of whether it is possible to give a version of the proof without this unexpected disruption of symmetry.

To look at it from another angle, the order in which the algorithm given in this proof operates on the terms of $f$ is not actually the lex order.  Rather it is the order that lex order induces on the set of \emph{conjugacy classes} of terms.  The first conjugacy class is the one containing the lexicographically leading term, the second contains the lexicographically highest-ranking term not contained in the first, etc.  We could call this \emph{symmetric lexicographic order}. Note that it is no longer a total order on the monomials (only on the conjugacy classes). Thus the proof's appeal to lex order is somehow deceptive. The real order is something else. From this angle, the pedagogically pregnant question becomes, are there descriptions of symmetric lex order that do not pass through actual lex order?

It was the sense of dissonance described here, between the participants' approach and the one taken in this classical proof, that led us to return to the idea of ``spread between the exponents" mentioned in the last section. This idea ultimately brought answers to the above questions, in the form of both an alternative proof, and a much richer understanding of the above proof.\footnote{It should be noted that many proofs of the FTSP are known, and they do not all share the surprising symmetry-breaking feel of the lex proof. Some of our favorites are the one found in \cite[~pp.9-12]{edwards} and the one found in \cite[~pp.204-5]{lang} and \cite[~pp.550-1]{artin}. In fact, one can derive the FTSP from Galois theory itself, rather than the reverse, because the modern development of the latter no longer depends on the former, as is done in \cite{hungerford}. The lex proof nonetheless stands out as especially (i) constructive, in that the algorithm it gives is practical for writing symmetric polynomials in terms of the elementary ones; (ii) short; and (iii) enduringly popular -- in addition to the citations above, see for example \cite{jacobson}, \cite{cohn}, \cite{tignol}, \cite{rotman}, and the original classic abstract algebra text, \cite{vanderwaerden}.}

\section{Spreadness proof}\label{spreadness}

We return to the ideas of our proof in the two and three variable cases and develop them into a complete argument.  Recall that to generalize our ideas, we first need to overcome the difficulty of deciding what the ``spread between the exponents'' means when there is a larger number of variables.  Indeed, finding this definition is the linchpin of our strategy.  We will give a definition (and later, a family of definitions) that allow us to prove the theorem by building an algorithm that picks off the monomials with the most spread-out exponents first.  The algorithm is identical in spirit and similar in practice to the standard one, but uses spread-out-ness (what we henceforth call ``spreadness") rather than lex order to determine which monomials to cancel out first. It thus carries out the classical proof's program while avoiding the symmetry disruption imposed by the lexicographic ordering (answering ``yes" to our first pedagogically resonant question above).

In Section~\ref{twoandthree}, our most na\"ive idea for defining ``spread'' was to use the highest exponent minus lowest.  Unfortunately, a simple computation shows this will not work in general. In terms of statistics, this is analogous to the range of the dataset of exponents of a given monomial $x_1^{i_1}\dots x_n^{i_n}$. But the range is not a good measure of dispersion because it does not involve all of the exponents. Instead we consider the following.

\begin{definition}
  Given a monomial $x_1^{i_1}\cdots x_n^{i_n}$ define its \textbf{spreadness} to be the sum $i_1^2+\dots+i_n^2$.
\end{definition}

Again in terms of statistics, this is equivalent to (in the sense that it induces the same ordering as) the variance of the dataset of exponents. The spreadness is also equivalent to the height of the center of gravity of the monomial when it is pictured as a pile of bricks, with a stack of $i_k$ bricks corresponding to each $x_k$.  (We will show this below.)  Moreover, it is a nonnegative integer, allowing us to use it as the basis of an induction.

The key fact to establish is that just as $c_1x_1^{i_1}x_2^{i_2}\dots x_n^{i_n}$, with $i_1\geq i_2\geq \dots \geq i_n$, is the leading term of $c_1\sigma_1^{i_1-i_2}\sigma_2^{i_2-i_3}\cdots \sigma_n^{i_n}$ when the terms are ordered lexicographically, it and all its conjugates also have strictly greater spreadness than the rest of the terms of this latter product.

\begin{theorem}[Spreadness Lemma]
  Given $i_1,\ldots,i_n$ with $i_1\geq i_2\geq \dots \geq i_n$, the terms of $\sigma_1^{i_1-i_2}\sigma_2^{i_2-i_3}\cdots \sigma_n^{i_n}$ with maximum spreadness are precisely $x_1^{i_1}x_2^{i_2}\cdots x_n^{i_n}$ and its conjugates.
\end{theorem}

\begin{proof}[Center of gravity proof]
  In this argument we identify a monomial $x_1^{j_1}x_2^{j_2}\ldots x_n^{j_n}$ with a sequence of stacks of heights $j_1,\ldots,j_n$ of identical bricks.  We first compute that for terms taken from $\sigma_1^{i_1-i_2}\sigma_2^{i_2-i_3}\cdots \sigma_n^{i_n}$, the spreadness is an increasing linear function of the vertical coordinate ($y$) of the center of gravity of its corresponding brick configuration.  Supposing that each brick has unit mass, then the vertical coordinate of the center of gravity is given by the sum over the bricks of each brick's height, divided by the number of bricks.  If we suppose the first brick of each stack lies at a height of $1$ and each brick has unit height, then the stack of height $j_1$ contributes $1+2+\dots+j_1=\frac{j_1(j_1+1)}{2}$ to the sum.  The full vertical coordinate $y$ of the center of gravity is then given by
  \begin{align*}
    y&=\frac{1}{d}\left(\frac{j_1(j_1+1)}{2}+\dots+\frac{j_n(j_n+1)}{2}\right)\\
    &=\frac{1}{2d}\left(j_1^2+\dots+j_n^2+j_1+\dots+j_n\right) \\
    &=\frac{1}{2d}\left(s+d\right)
  \end{align*}
  where $d$ is the number of bricks (i.e.,\ the degree), and $s$ is the spreadness.  So $s=2dy-d$ and since $d$ is fixed, $s$ is an increasing linear function of $y$ as claimed.

  Next, we observe that all of the terms of $\sigma_1^{i_1-i_2}\sigma_2^{i_2-i_3}\cdots \sigma_n^{i_n}$ can be obtained from $x_1^{i_1}x_2^{i_2}\cdots x_n^{i_n}$ by moving some bricks horizontally (and dropping them onto the top of the stack below if necessary).  The conjugates of $x_1^{i_1}x_2^{i_2}\cdots x_n^{i_n}$ are precisely those terms in which each layer of bricks rests completely on top of the layer below it before any dropping takes place.  Thus bricks will fall for precisely those terms that are not conjugates of $x_1^{i_1}x_2^{i_2}\cdots x_n^{i_n}$.  See Figure~\ref{fig:bricks}.

  Finally, we appeal to the simple fact that given any physical configuration of bricks, moving some bricks to lower positions decreases the center of gravity.
\end{proof}

\begin{figure}[h]
\def\brick{rectangle +(.95,.45)}
\begin{tikzpicture}[scale=.8]
  \foreach \x in {1,...,5} \node[anchor=south] at (\x-.5,-.6) {$x_\x$};
  \foreach \x in {0,1,2,3} \draw (\x,0) \brick;
  \foreach \x in {0,1,2} \draw (\x,.5) \brick;
  \draw (0,1) \brick;
  \draw (0,1.5) \brick;
  \draw (0,2) \brick;
\end{tikzpicture}
\hspace{1cm}
\begin{tikzpicture}[scale=.8]
  \foreach \x in {1,...,5} \node[anchor=south] at (\x-.5,-.6) {$x_\x$};
  \foreach \x in {0,2,3,4} \draw (\x,0) \brick;
  \foreach \x in {0,1,4} \draw (\x,.5) \brick;
  \draw (0,1) \brick;
  \draw (3,1.5) \brick;
  \draw (4,2) \brick;
\end{tikzpicture}
\hspace{1cm}
\begin{tikzpicture}[scale=.8]
  \foreach \x in {1,...,5} \node[anchor=south] at (\x-.5,-.6) {$x_\x$};
  \foreach \x in {0,1,2,3,4} \draw (\x,0) \brick;
  \foreach \x in {0,3,4} \draw (\x,.5) \brick;
  \foreach \x in {0,4} \draw (\x,1) \brick;
  \draw[->,dotted] (1.5,.75) -- (1.5,.25);
  \draw[->,dotted] (3.5,1.75) -- (3.5,.75);
  \draw[->,dotted] (4.5,2.25) -- (4.5,1.25);
\end{tikzpicture}
\caption{\emph{Left}: the target term $x_1^5x_2^2x_3^2x_4$.  \emph{Center}: another generic term from the product $\sigma_1^3\sigma_3\sigma_4$; in this picture the term $x_1^3x_2x_3x_4^2x_5^3$ is represented.  \emph{Right}: the same generic term with the bricks ``fallen;'' it has a lower center of mass than the target term.\label{fig:bricks}}
\end{figure}
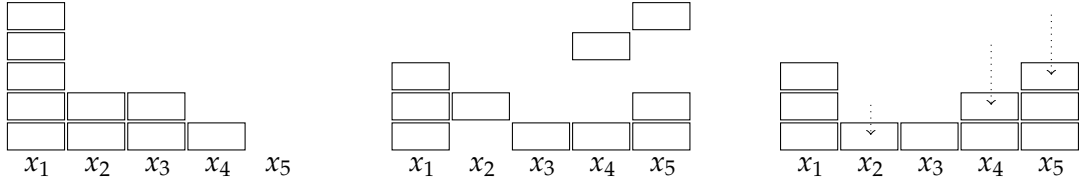

Once this is established, the proof of the fundamental theorem follows the outline of the standard argument given above.

\begin{proof}[Proof of the FTSP using the Spreadness Lemma]
  Let $f$ be the symmetric polynomial to be represented.  As above, we lose nothing by assuming $f$ is homogeneous.

  The algorithm proceeds as in the standard proof except with spreadness playing the role of lexicographic order.  Pick \emph{any} term of $f$ with maximum spreadness $s_1$ and consider it and its conjugates.  Form the product of elementary symmetric polynomials $g_1$ that has these terms as its terms of maximum spreadness.  (If the terms of $f$ have coefficient $c_1$ and exponents $i_1 \geq i_2 \geq \dots \geq i_n$, then $g_1=c_1\sigma_1^{i_1-i_2}\sigma_2^{i_2-i_3}\dots\sigma_n^{i_n}$ as always.)  Then since these terms are the only terms of $g_1$ with spreadness as high as $s_1$ by the Spreadness Lemma, $f-g_1$ contains fewer terms of spreadness $s_1$ than $f$ does, possibly zero.

  Continuing in like manner beginning with $f-g_1$, forming $g_2$ and then $f-g_1-g_2$, etc., we get an algorithm that must terminate because at each stage, either the maximum spreadness or the number of terms with this spreadness has been decreased.

  The uniqueness of the representation follows exactly as it did in the standard proof.  Distinct products of elementary symmetric polynomials will have distinct terms of maximum spreadness because of the injectivity of the map $(k_1,\dots,k_n) \mapsto (k_1+\dots+k_n,\dots,k_n)$.  Therefore complete cancellation is impossible:  any nonzero polynomial in the elementary symmetric polynomials will be nonzero when multiplied out.
\end{proof}

As an aside, we mentioned above that spreadness is also equivalent to variance.  To see this, we compute that for terms taken from $\sigma_1^{i_1-i_2}\sigma_2^{i_2-i_3}\cdots \sigma_n^{i_n}$, the spreadness $s$ is an increasing linear function of the variance $\sigma^2$ of the set $\left\{j_1,j_2,\ldots,j_n\right\}$.  Indeed,
  \begin{align*}
    \sigma^2 &= \frac1n\left(j_1^2+\cdots+j_n^2\right)-\mu^2
  \end{align*}
  Here $n$ is fixed and so is the mean $\mu$, being a function of just $n$ and the degree $d$.  Thus, $s=n\sigma^2+n\mu^2$ is an increasing linear function of $\sigma^2$.

\section{The spreadness and lex orderings}\label{spreadnessandlex}

It is natural to ask whether there is any relationship between the spreadness and lexicographic orderings on monomials. A propos of our discussion at the end of Section~\ref{classical}, the more natural comparison is between spreadness and what we there defined as \emph{symmetric lexicographic order}, i.e.\ the order that lex induces on conjugacy classes of monomials. In the Spreadness Lemma above, we have shown that the two orderings single out the same conjugacy class of monomials as leading among those that occur in a single product of the form $\sigma_1^{i_1-i_2}\sigma_2^{i_2-i_3}\cdots \sigma_n^{i_n}$.

In general, however, the two orderings do not agree. For example, $x_1^3x_2x_3x_4x_5x_6$ beats $x_1^2x_2^2x_3^2x_4^2$ lexicographically, but has a lower spreadness with a score of 14 versus 16.

Still, this can be remedied by replacing the spreadness with the $p^\text{th}$ moment (that is, $i_1^p+\cdots+i_n^p$) for suitably large $p$.\footnote{We would like to thank Walter Stromquist for this observation.} In the above examples, letting $p=3$, the new score becomes 86 versus 64.

To see that this can be done generally, let $x_1^{i_1}\cdots x_n^{i_n}$ and $x_1^{j_1}\ldots x_n^{j_n}$ be given with $i_1\geq\cdots\geq i_n$ and $j_1\geq\cdots\geq j_n$. Assume that $x_1^{i_1}\cdots x_n^{i_n}$ precedes $x_1^{j_1}\ldots x_n^{j_n}$ lexicographically, and let $k$ be the least such that $i_k>j_k$. Then we may choose $p$ large enough that $i_k^p>nj_k^p$, and it follows easily that $i_1^p+\cdots+i_n^p>j_1^p+\cdots+j_n^p$. This proves:

\begin{theorem}\label{symlimit}
Symmetric lex order is the limit of the order on conjugacy classes of monomials given by the $p^\text{th}$ moment as $p\to\infty$, in the sense that given any finite set of classes, for all sufficiently high $p$ the $p^\text{th}$ moment order on those classes matches the symmetric lex order.
\end{theorem}

This provides an answer to our second mathematical-but-pedagogically-motivated question from the end of Section~\ref{classical}: a way to characterize symmetric lex order without passing through lex order. To our taste, this characterization shows that symmetric lex order is ``more natural" than is obvious from its definition via (actual) lex order.

Moreover it is possible to give a version of the Spreadness Lemma for any of the higher moments, although the proof is somewhat more involved without the center-of-gravity interpretation available.

\begin{theorem}[Spreadness Lemma for higher moments]
  Given $i_1,\ldots,i_n$ with $i_1\geq i_2\geq \dots \geq i_n$, the terms of $\sigma_1^{i_1-i_2}\sigma_2^{i_2-i_3}\cdots \sigma_n^{i_n}$ with maximum $p^\text{th}$ moment, for $p>1$, are precisely $x_1^{i_1}x_2^{i_2}\cdots x_n^{i_n}$ and its conjugates.
\end{theorem}

\begin{proof}[Proof outline]
  The terms $x_1^{j_1}x_2^{j_2}\cdots x_n^{j_n}$ of $\sigma_1^{i_1-i_2}\sigma_2^{i_2-i_3}\cdots \sigma_n^{i_n}$ all satisfy the following conditions: every exponent $j_k$ is $\leq i_1$, every sum of two exponents $j_k+j_{k'}$ is $\leq i_1+i_2$, and more generally every sum of $l$ many exponents is $\leq i_1+\cdots+i_l$, with equality for $l=n$. Thus each term corresponds to a lattice point $(j_1,j_2,\dots,j_n)$ in the first quadrant of $\mathbb{R}^n$, contained in the convex polytope $P$ cut out by the inequalities
  \begin{align*}
    z_k&\leq i_1,\;\forall k\in [n]\\
    z_k+z_k'&\leq i_1+i_2, \;\forall k,k'\in [n]\\
       &\:\;\vdots \\
    z_{k_1}+\cdots + z_{k_l} &\leq i_1+\cdots+i_l, \;\forall k_1,\dots,k_l\in [n]\\
       &\;\;\vdots\\
    z_1+\cdots+z_n &= i_1+\cdots+i_n
  \end{align*}
  Furthermore, the term $x_1^{i_1}x_2^{i_2}\cdots x_n^{i_n}$ and its conjugates correspond exactly to those lattice points that realize equality in each of the above inequalities for some choice of $k$'s.  In other words, to \emph{the vertices of the convex polytope $P$}. This is because, in the language and imagery of the center of gravity proof of the Spreadness Lemma, equality is realized in each inequality (for a maximizing choice of $k$'s) if and only if no brick has fallen.  If a brick in the $l$th highest stack falls to a lower stack, this implies that the highest $l$ stacks now have a lower total than they did originally.
  
  
  
  Now we appeal to the fact that the $p^\text{th}$ moment is a monotone function of the $L^p$ norm on $\mathbb{R}^n$, and for $p>1$ this norm is strictly subadditive, i.e. equality holds in $\|u+v\| \leq \|u\|+\|v\|$ only when one of $u,v$ is a nonnegative multiple of the other.  It follows that if $\|u\|=\|v\|$ and $u\neq v$, any nontrivial convex combination of $u,v$ has strictly smaller norm than $u,v$ have. (One sees this by applying the inequality to $\|\mu u + \nu v\|$ with $\mu,\nu>0$ and $\mu+\nu=1$.)  More generally, if the extreme points of a convex body all have the same norm, all the other points of the body must have strictly smaller norm.  In the present case, the tuples $(j_1,\dots,j_n)$ corresponding to $x_1^{i_1}\cdots x_n^{i_n}$ and its conjugates all have the same $L^p$ norm $\sqrt[p]{i_1^p+\cdots+i_n^p}$. Since they are the vertices of a convex polytope containing the tuples corresponding to all the other terms, these latter must have smaller $L^p$ norm and therefore smaller $p^\text{th}$ moments. 
\end{proof}
  
Thus, the FTSP can be proven using the $p^\text{th}$ moment for \emph{any} $p>1$. The spreadness proof given in Section~\ref{spreadness} is only the ``first'' in an infinite sequence of nearly identical proofs, and the classical proof in Section~\ref{classical} is, by Theorem~\ref{symlimit}, in some sense the ``last.'' All the proofs have in common an algorithm that represents an arbitrary symmetric polynomial $f$ by forming products of elementary symmetric polynomials $\sigma_k$ in such a way as to mimic $f$'s terms of maximum exponent dispersion first; thus they are all fundamentally inductions on the extent of exponent dispersion---the namesake of our notion of spreadness. Each proof measures exponent dispersion a little differently but they all agree about the terms of maximum dispersion in expansions of monomials in the $\sigma_k$'s. They all agree because these terms correspond to the extreme points of certain convex polytopes in $\mathbb{R}^n$, although we have other, easier ways to see this in the ``first'' and ``last'' cases. Since the order in which the classical algorithm operates on $f$ comes from the limit of these ways of measuring, we can see it as really having in some sense been measuring exponent dispersion all along! This is the ``richer understanding'' of the classical proof promised at the end of Section~\ref{classical}.

\section{The FTSP in Galois' work}\label{ftspingalois}

In this concluding section we place the FTSP in the greater context of Galois theory by showing how it is an example of a larger phenomenon.  The FTSP says that expressions that are \emph{completely} symmetric are \emph{completely} rationally expressible.  In his seminal essay \emph{M\'{e}moire sur les conditions de r\'{e}solubilit\'{e} des \'{e}quations par radicaux}, Galois proved a series of results that tie types of partial symmetry to types of partially rational expressibility as well.  First, we justify our until-now flip use of phrases like ``rationally expressible'' (since the FTSP is only a statement about polynomials; no division allowed) extending the FTSP to rational functions:

\begin{theorem}[FTSP for rational functions]\label{ftsprational}
  Any rational function in $x_1,\ldots,x_n$ that is symmetric in $x_1,\ldots,x_n$ is a rational function of the elementary symmetric polynomials $\sigma_1,\ldots,\sigma_n$.
\end{theorem}

\begin{proof}
Let $f$ be such a function.  It is a quotient of polynomials $f=P/Q$.  Let $Q',Q'', \dots, Q^{(n!)}$ be the result of permuting the variables in $Q$ in every possible way.  Then
\[
f=\frac{PQ'Q''\cdots Q^{(n!)}}{QQ'Q''\cdots Q^{(n!)}}
\]
The denominator of this expression is invariant under all permutations of the $x_i$'s by construction, and $f$ is as well by assumption.  It follows that the numerator is also invariant (symmetric).  Thus $f$ is here expressed as a quotient between symmetric polynomials, which are polynomials in the $\sigma_i$ by the FTSP.
\end{proof}

What Galois did was to reveal the FTSP as just the first link in a chain of statements that tie types of symmetry to forms of rational expressibility. We give them without proof. The next chain link was already well-known in Galois' time:

\begin{theorem}
  \label{point-stabilizer}
  If $f$ is a rational function of $x_1,\dots,x_n$ that is symmetric under all permutations of the $x_i$'s that \emph{fix} $x_1$, then it is expressible as a rational function of $\sigma_1,\dots,\sigma_n$ \emph{and} $x_1$.
\end{theorem}

But the next is due to Galois and appears as Lemma III in his essay:

\begin{theorem}
  \label{galoisresolvent}
  If $V$ is a rational function of $x_1,\dots,x_n$ that is not left fixed by \emph{any} nontrivial permutation of the $x_i$'s, then \emph{every} rational function of the $x_i$'s is expressible as a rational function of the $\sigma_i$'s and $V$.
\end{theorem}

We can summarize the connections between symmetry and rational expressibility in a table as follows.
\begin{center}
\begin{tabular}{l|l}
If it is invariant under\ldots & then it is rationally expressible in\ldots\\
\hline \quad all permutations & \quad $\sigma_1,\dots,\sigma_n$\\
\quad all permutations that fix $x_1$ & \quad $\sigma_1,\dots,\sigma_n$ and $x_1$\\
\quad any subset, or no permutations at all & \quad $\sigma_1,\dots,\sigma_n$ and $V$
\end{tabular}
\end{center}

This leads us to the statement of Galois' famous Proposition~I. Theorems \ref{ftsprational}, \ref{point-stabilizer}, and \ref{galoisresolvent} are all simultaneously lemmas for and special cases of this grand result, which forms one half of what is now called the \emph{Fundamental Theorem of Galois Theory}. The following paraphrases Galois' statement, as his original language presupposes some conventions he set up previously.

\begin{PropositionI}
  Let $f$ be a polynomial with coefficients $\sigma_1,\dots,\sigma_n$.  Let $x_1,\dots,x_n$ be its roots.\footnote{Up until now, the $x_i$ have been formal symbols, and the $\sigma_i$ have been formal polynomials in them, but for this statement the $\sigma_i$ are prior to the $x_i$ and may be elements of any field containing $\mathbb{Q}$. Galois tacitly assumed that the roots $x_i$ of $f$ \emph{exist}, somewhere, in some sense. Today we would say he assumes the existence of a \emph{splitting field}. Most mathematicians prior to the nineteenth century working in algebra made this same assumption without question. Gauss famously argued that it needed justification, in motivating his proof of the so-called \emph{Fundamental Theorem of Algebra}, that every integer polynomial splits into linear factors over $\mathbb{C}$. See \cite[pp.~912--3]{edwards12}.}\label{splittingfield}  Let $U,V,\ldots$ be some numbers that are rational functions of the $x_i$'s that we suppose are known to us.  Then, there exists a group $G$ of permutations of the $x_i$'s such that the rational functions of the $x_i$'s fixed under all the permutations in $G$ are exactly those whose values are rationally expressible in terms of $\sigma_1,\dots,\sigma_n$ and $U,V,\ldots$\footnote{It is worth mentioning a possible source of ambiguity pointed out by Edwards in \cite{edwards12}. It makes sense to speak of a permutation acting on a rational function of a set of formal unknowns, by permuting the unknowns. It doesn't make sense to speak of it acting on a numerical value---what's being permuted? Galois is talking about something in between: a rational function of \emph{roots of $f$}. It thus may be viewed both as a formal object (treat each root as an unknown, so we can apply a permutation to them) and a numerical object (evaluate the rational function on the given roots). When he speaks of applying a permutation, he is taking the formal viewpoint. But when he speaks of the function being fixed, he is speaking of its numerical value, as he makes clear in a footnote. One must imagine permuting the unknowns in a formal rational function of unknowns, and then obtaining a numerical value by substituting the given roots for the unknowns and evaluating the expression obtained thereby. Thus a clearer statement of the theorem would be: there exists a group $G$ of permutations such that a rational function evaluates to the same numerical value after each of the permutations of $G$ has been applied to its unknowns if and only if this same numerical value is rationally expressible in terms of $\sigma_1,\dots,\sigma_n$ and $U,V,\dots$.}
\end{PropositionI}

If you have studied Galois theory, this formulation may feel unfamiliar to you.  To see that it is really the same thing you have seen before, consider that the set of quantities that are rational functions of $\sigma_1,\dots,\sigma_n$ forms a field ($f$'s \emph{coefficient field}); similarly for the set of quantities that are rational functions of $x_1,\dots,x_n$ ($f$'s \emph{splitting field}).  The set of rational functions of $\sigma_1,\dots,\sigma_n,U,V,\dots$ is some extension of the coefficient field contained in the splitting field. So we can state Galois' Proposition~I in the following modern  way: Given a polynomial $f$ and a field $K$ lying between $f$'s coefficient field and $f$'s splitting field, there exists an action of some permutation group $G$ on the roots of $f$ which extends to an action on the splitting field of $f$ and such that the fixed field of this action is exactly $K$.\footnote{More succinctly, every intermediate field corresponds to some group. The other half of what we now call the Fundamental Theorem of Galois Theory, alluded to above, states that if you find the group $G$ corresponding to the coefficient field itself (which is called the \emph{Galois group of $f$}), then every subgroup of $G$ corresponds to some intermediate field $K$. There is thus a one-to-one correspondence between fields intermediate between $f$'s coefficient and splitting fields, on the one hand, and subgroups of $f$'s Galois group, on the other.} This powerful and famous result ties in a very precise way that which is rationally expressible (the elements of a field) to a given type of symmetry (the group).

We hope to have shown you that the FTSP contains the first whisper of this connection.  If you are interested to learn more, Harold Edwards' 2012 article in the \emph{Notices of the AMS} \cite{edwards12} explicates some of Galois' own proofs of the above propositions in modern language.  This article is perhaps best appreciated by reading it alongside Galois' original essay, which is printed in English translation in several sources (see for instance \cite{edwards}, \cite{hawking}, or the recent, particularly comprehensive \cite{neumann}).

\section{Addendum (by Ben Blum-Smith)}

After we completed work on this article, I learned of the following beautiful paper: 

Garsia, A. (1980). Combinatorial methods in the theory of Cohen-Macaulay rings. {\em Advances in Mathematics}, 38:229-266.

Garsia's work contains a precedent for the ``brick-stacking" argument found in the present article. In the context of giving explicit descriptions of certain rings as free modules over certain subrings, Garsia used a diagram much like our Figure 1 to argue a point much like our Theorem 3. See especially Figures 1 and 2 on pp. 257-258 of Garsia's paper, and the proof of Lemma 6.1 on p. 258. The context is different but the core insight is the same.

\bibliographystyle{apalike}
\bibliography{symmetric}

\end{document}